\documentclass[a4paper,11pt]{amsart}
\usepackage[foot]{amsaddr}
\usepackage{geometry}
\usepackage{amsfonts}
\usepackage{amsthm}
\usepackage{amssymb}
\usepackage{amsmath}
\usepackage{comment}
\usepackage{theoremref}
\usepackage{enumerate}
\usepackage{mathtools}
\usepackage{a4wide}
\usepackage{xcolor}
\usepackage[hidelinks]{hyperref}
\usepackage{graphicx}
\makeatletter
\newcommand{\proofpart}[2]{%
    \par
  \addvspace{\medskipamount}%
  \noindent\emph{Step #1: #2}\par\nobreak
  \addvspace{\smallskipamount}%
  \@afterheading
}
\makeatother

\DeclarePairedDelimiter\abs{\lvert}{\rvert}%
\DeclarePairedDelimiter\norm{\lVert}{\rVert}%

\makeatletter
\let\oldabs\abs
\def\abs{\@ifstar{\oldabs}{\oldabs*}}
\let\oldnorm\norm
\def\norm{\@ifstar{\oldnorm}{\oldnorm*}}
\makeatother

\makeatletter
\g@addto@macro\bfseries{\boldmath}
\makeatother

\newcommand{\A}{\mathcal{A}}

\newcommand{\T}{\mathbb{T}}

\newcommand{\conj}[1]{\overline{#1}}
\newcommand{\D}{\mathbb{D}}
\newcommand{\B}{\mathcal{B}}

\newcommand{\cD}{\conj{\mathbb{D}}}

\renewcommand\Re{\operatorname{Re}}

\newtheorem{thm}{Theorem}[section]
\newtheorem{lemma}[thm]{Lemma}

\newtheorem{cor}[thm]{Corollary}
\newtheorem{prop}[thm]{Proposition}

\theoremstyle{definition}

\theoremstyle{definition}

\newtheorem{remark}[thm]{Remark}

\begin{document}
\title{Invertible non-cyclic functions in the little Bloch space}

\author{ Adem Limani}
\thanks{The author acknowledges financial support from the Swedish Research Council (VR), Starting Grant \mbox{2025-05093}.} 
\address{Centre for Mathematical Sciences, Lund University, Sweden}
\email{adem.limani@math.lu.se}

\date{\today}

\begin{abstract}
This paper presents a method for exhibiting functions $F$ in the little Bloch space $\B_0$ which are invertible, in the sense that $1/F$ also belongs to $\B_0$, and yet fail to be cyclic for the forward shift. The functions obtained in this way give rise to shift-invariant subspaces in $\B_0$ of index two, of a special type. The approach rests on a harmonic-measure criterion for non-cyclicity. This criterion gives a new manifestation of a principle of Borichev and Hedenmalm, and applies in a broad class of spaces of analytic functions.
\end{abstract}
\maketitle

\section{Introduction}
\subsection{Background on invertibility and cyclicity}

Let $X$ be a Fr\'echet space of analytic functions on
the unit disc $\D$, in which analytic polynomials are
dense and the multiplication operator
\[
 M_zf(z)=zf(z)
\]
acts continuously on $X$. A function $f\in X$ is said to be \emph{invertible} in $X$ if $1/f\in X$, 
while a function $f \in X$ is said to be \emph{cyclic} (with respect to $M_z$) if the linear span of 
\[
\left\{f, M_z f, M^2_z f, \dots \right\}
\]
forms a dense subspace of $X$. Under the assumption that the analytic polynomials are dense in $X$, cyclicity is equivalent to the existence of analytic polynomials $Q_n$ such that
\[
 Q_n f \to 1 \quad  \text{ in} \quad X.
\]
Cyclicity can therefore be viewed as an approximate form of invertibility, traditionally referred to as \emph{weak invertibility}.

In the classical theoretical frameworks of commutative Banach algebras $X$ with units, equipped with pointwise multiplication, the notions of invertibility and cyclicity coincide. Indeed, invertibility
implies cyclicity by polynomial approximation of $1/f$, and conversely, if $p_nf\to1$ in $X$, then by openness property of resolvents, $p_nf$ is invertible for sufficiently large $n$, and hence so is $f$.

The relationship between invertibility and cyclicity takes a different turn when one only slightly moves away from the framework of Banach algebras. To clarify this point, consider the Hardy spaces $H^p$ of analytic functions $f$ in $\D$ satisfying
\[
 \norm{f}_{H^p}^p
 :=\sup_{0<r<1}\int_\T |f(r\zeta)|^p\,dm(\zeta)<\infty,
 \qquad 1\le p<\infty,
\]
where $dm$ is normalized arc-length measure on $\T$.
A classical theorem of Beurling implies that the cyclic functions in $H^p$ are precisely the outer functions. Therefore, one can easily show that if both $f$ and $1/f$ belong to $H^p$, then $f$ is outer, and hence cyclic. However, the converse is not true, as 
\[
f(z)=1-z
\]
is a bounded outer function, whereas $1/f\notin \bigcup_{p\geq 1} H^p$. Therefore, the concept of invertibility becomes strictly stronger than cyclicity in the setting of Hardy spaces. The implication from invertibility to cyclicity also holds
in the framework of Korenblum growth space
\[
 \A^{-\infty}
 :=\bigcup_{\alpha>0}
 \left\{f\in\operatorname{Hol}(\D):
 \sup_{z\in\D}(1-|z|)^\alpha|f(z)|<\infty\right\},
\]
equipped with its natural inductive limit topology.
Since $\A^{-\infty}$ is an algebra with dense polynomials, polynomial approximation of $1/f$ again implies cyclicity, whereas the converse fails. For instance, see \cite{KorenblumCyc}. The more delicate question is whether invertibility implies cyclicity when looking at a fixed growth space. To this end, we define the growth space of order $\alpha>0$ as the Banach space of analytic functions $f$ in $\D$ satisfying
\[
 \A_0^{-\alpha}
 :=\left\{f\in\operatorname{Hol}(\D):
 (1-|z|)^\alpha|f(z)|\to0
 \text{ as }|z|\to1\right\},\qquad \alpha>0,
\]
equipped with the norm
\[
 \norm{f}_{\A^{-\alpha}}
 :=\sup_{z\in\D}(1-|z|)^\alpha|f(z)|.
\]

The invertibility problem in these spaces is closely related to the corresponding problem in the Bergman spaces $A^p$,
defined by
\[
 \norm{f}_{A^p}^p
 :=\int_\D |f(z)|^p\,dA(z)<\infty,
 \qquad 1\le p<\infty,
\]
where $dA$ denotes the Lebesgue area measure. In both settings, the question whether invertibility implies cyclicity was a longstanding problem. Influential results of Brown and Korenblum \cite{KorenblumBrown} proved that if we impose stronger initial conditions on $f$, then invertibility implies cyclicity. More specifically, for every sufficiently small $\varepsilon>0$, we have
\[
 \begin{aligned}
 f\in\A_0^{-(\alpha-\varepsilon)},\quad 1/f\in\A_0^{-\alpha}
 &\quad\Longrightarrow\quad
 f\text{ is cyclic in }\A_0^{-\alpha},\\
 f\in A^{p+\varepsilon},\quad 1/f\in A^p
 &\quad\Longrightarrow\quad
 f\text{ is cyclic in }A^p.
 \end{aligned}
\]
See \cite[Chapter~7]{HKZ2000} for detailed treatments of similar problems. The question was whether one can simply take $\varepsilon=0$ in the above statements. Borichev and Hedenmalm
\cite{borichev1995cyclicity,borichevhedenmalmcyclicity}
showed that one cannot, by constructing invertible
non-cyclic functions in both $\A_0^{-\alpha}$ and in $A^p$.
Borichev, Hedenmalm, and Volberg
\cite{borichev2004large} subsequently obtained analogous
examples in a broad class of weighted Bergman spaces,
and used them to construct invariant subspaces of index
greater than one.

\subsection{Main results in the little Bloch space}
Our main intention is to consider the relationship between invertibility and cyclicity in a framework which is substantially tighter than that of Bergman spaces.
Let $\B$ denote the Banach space of analytic functions $f$ in $\D$ equipped with the norm
\[
 \norm{f}_{\B}=\abs{f(0)}+
 \sup_{z\in\D}(1-\abs{z})\abs{f'(z)}< \infty.
\]
As $\B$ is not separable (with respect to the norm topology), we consider the closure of the analytic polynomials in $\B$, which gives rise to the little Bloch space $\B_0$ consisting of analytic functions $f$ in $\D$ with 
\[
\lim_{|z|\to 1^-} (1-|z|)f'(z)=0.
\]
The Bloch space may be regarded as a limiting case of the Bergman spaces $A^p$ as $p\to \infty$, in the same way as $\mathrm{BMOA}$ is the limiting case of the Hardy spaces $H^p$. Meanwhile, it may also be regarded as a limiting case of the algebras $\Lambda_a^\alpha$ of analytic functions in $\D$ with $\alpha$-H\"older continuous extensions to $\cD$. Indeed, by a classical theorem of Hardy and Littlewood (see, for instance, \cite{DurenHardy}), the functions in $\Lambda_a^\alpha$ can be characterized by the properties:
\[
|f'(z)| = O\left( (1-|z|)^{\alpha-1} \right), \qquad |z| \to 1^-,
\]
hence the Bloch space may also be regarded as a limiting case as $\alpha \to 0^+$.

It was recently proved in \cite{LimaniNicolau} that there exist functions $f \in \B$ with $1/f \in \B$, yet $f$ is not cyclic in $\B$, now equipped with the weak-star topology. This addressed a problem that remained open in \cite{anderson1991inner} since the early 1990s. The construction of such functions was based on special families of universal covering mappings, which are well-known to never belong to $\B_0$, and could therefore not be used in the setting of $\B_0$. Difficulties of a similar kind, highlighting the discrepancy between establishing weak-star cyclicity in the growth space $A^p$ and in the little-o growth space $\A^{-\alpha}_0$ were also mentioned and clarified in the work of Borichev and Hedenmalm in \cite{borichevhedenmalmcyclicity}.

The main purpose of this paper is to present a methodology which yields a stronger result in $\B_0$, which in addition yields the existence of very particular $M_z$-invariant subspaces in $\B_0$ of multiple index. To this end, let $[f_1,\ldots,f_n]$ denote the closed shift-invariant subspace
generated by $f_1,\ldots,f_n\in\B_0$, and recall that its index is defined as the number
\[
 \operatorname{ind}\mathcal M=\dim \mathcal M/z\mathcal M.
\]
The main result goes as follows.

\begin{thm}\thlabel{main:THM} There exists an analytic function $F$ in $\D$ which satisfies:
\begin{enumerate}
    \item[(i)] $F$ and $1/F$ belong to $\B_0$,
    \item[(ii)] $[F,F^{-1}]$ has index $2$.
\end{enumerate}
\end{thm}

We emphasize that the principal point of Theorem~\ref{main:THM} is not to exhibit $M_z$-invariant subspaces of large index, as this has previously been done in the recent work of Biehler \cite{BIEHLER2025111034}. The striking point is that Theorem~\ref{main:THM} gives rise to $M_z$-invariant subspaces of a very specific kind, which lend themselves to establishing the following remarkable consequence.

\begin{cor} \thlabel{COR:CYCINV} There exists an invertible function $F$ in $\B_0$ which is not cyclic in $\B_0$.
\end{cor}

Indeed, neither $F$ nor $F^{-1}$ in the statement of Theorem~\ref{main:THM} can be cyclic in $\B_0$, since otherwise $[F,F^{-1}]=\B_0$.

We remark that the polynomial (outer) function $f(z)=1-z$ is easily seen to be cyclic in $\B_0$, but $1/f \notin \B_0$. Therefore the relationship between invertibility and cyclicity is lost in the framework of $\B_0$, sharing a similar trait to the settings of Bergman spaces and fixed growth spaces.

\paragraph{\textbf{Generalizations to weighted Bloch spaces.}}
Let $w:(0,1]\to(0,1]$ be nondecreasing, with $w(t)/t^{1-\delta}$ nonincreasing for some $\delta>0$, and let $\B_0(w)$ denote the space of analytic functions $g$ in $\D$ with
\[
 \lim_{|z|\to1^-}\frac{(1-|z|)\,|g'(z)|}{w(1-|z|)}=0.
\]
These spaces are contained in $\B_0$, and the square means of their elements grow at most like
\[
 \Phi_w(r)=1+\int_{1-r}^1w(t)^2\,\frac{dt}{t}.
\]
The methods of this paper carry over to $\B_0(w)$ whenever $\Phi_w(r)\to\infty$ as $r\to1^-$, and yield an invertible function $F\in\B_0(w)$ with $\operatorname{ind}[F,F^{-1}]=2$. In particular, $F$ is not cyclic in $\B_0(w)$. For instance, this applies to the weights $w(t)=\log^{-s}(e/t)$ with $0\le s\le1/2$. The necessary modifications are described in Remark~\ref{REM:weighted}, after the proof of Theorem~\ref{main:THM}.

The divergence of $\Phi_w$ is necessary for this approach, since $\B_0(w)\subset H^2$ when $\int_0^1 w(t)^2\,dt/t<\infty$, and the large values required by the criterion then cannot occur. We believe that the threshold is sharp, so that every invertible function in $\B_0(w)$ is then cyclic. This is elementary when $\int_0^1 w(t)\,dt/t<\infty$, since $\B_0(w)$ is then a Banach algebra, but the intermediate range appears to be considerably more subtle, and will be treated elsewhere.

\subsection{Outline of the proof}
The proof of Theorem~\ref{main:THM} is constructive, and differs from the approach in \cite{LimaniNicolau}, which produced an invertible function in $\B$ that is not weak-star cyclic. We look for $F$ of the form
\[
 F(z)=e^{f(z)},\qquad z\in\D,
\]
where $f$ is odd. Then $F(-z)=F(z)^{-1}$, so membership of $F$ in $\B_0$ automatically gives membership of $1/F$. The task is to keep $F$ in $\B_0$, while making it large on sets that obstruct cyclicity. The proof has two parts.
\begin{enumerate}
 \item[(i)] \emph{A criterion.} In Section~2, a new criterion for non-cyclicity is established, which asks that the large values of $F$ on suitable circles occupy sets of uniformly positive harmonic measure at every scale. It is guided by the Borichev--Hedenmalm principle: \emph{A function which has near-maximal growth on a sufficiently large set cannot be cyclic}, but rests on a different mechanism. Combined with the symmetry between $F$ and $F^{-1}$, it yields $\operatorname{ind}[F,F^{-1}]=2$, see Theorem~\ref{THM:SUFFNONCYCLIC}. It turns out that the criterion admits a far-reaching generalization, which gives a new and stronger manifestation of the aforementioned principle in a broad class of spaces, detecting index as well as non-cyclicity, see Section~\ref{SEC:general}.
 \item[(ii)] \emph{A construction.} In Sections~3 and~4, $f$ is built from long blocks of odd polynomial terms at widely separated frequencies, damped so that their effect on $e^{\pm f}$ remains small in Bloch norm. Decaying amplitudes keep $F$ in $\B_0$, and the length of the blocks produces the large values. In Section~\ref{SEC:GUIDE}, the construction is outlined in more detail.
\end{enumerate}
The two parts are linked by the Paley--Zygmund inequality, see Lemma~\ref{paley-zygmund}. Bounds on the second and fourth moments of $F$ on a circle force its large values to occupy a set of uniformly positive harmonic measure, without prescribing the geometry of that set. A probabilistic interpretation of the criterion is given in Remark~\ref{REM:brownian}. 

In Makarov's probabilistic picture of Bloch functions \cite{makarov1985,pommerenke1992}, $F$ behaves like a martingale over dyadic scales, and the moment bounds amount to a central limit theorem, with variance nearly as large as membership in $\B_0$ allows. In $\B_0(w)$ this variance is at most $\Phi_w(r)$, which is why the method requires $\Phi_w(r)\to\infty$.

\medskip
\paragraph{\textbf{Notation.}}
We write $A\lesssim B$ if $A\le cB$ for some constant $c>0$. Constants may change from line to line, and are absolute unless otherwise indicated.

\subsection{Acknowledgments and declarations}
The author is grateful to Artur Nicolau for fruitful discussions.

The research program leading to Corollary~\ref{COR:CYCINV} began in summer 2024, with the aim of constructing a non-cyclic function in $\B_0$ of the form
\[
 F(z)=\exp\Bigl(\sum_n a_nz^{\lambda_n}\Bigr),\qquad z\in\D,
\]
with odd frequencies $\lambda_n$ and Hadamard gaps, so that $F(-z)=F(z)^{-1}$. The difficulty was that, due to the lacunarity of the frequencies, functions in $\B_0$ of this form are too well behaved to have near-maximal growth on large sets. This motivated replacing individual monomials by blocks of odd polynomials, supported on lacunary and widely separated parts of the Fourier spectrum, so that different blocks barely interact.

ChatGPT Astra 6.0 assisted in developing a harmonic-measure criterion suited to this construction, and devised a formal proof. Claude Opus 5.5 assisted with the final revision of the exposition, including Remark~\ref{REM:brownian} and Figure~\ref{FIG:visibility}. The author formulated the research program and developed the exposition through successive revisions, with particular attention to the parameter choices and the mechanisms behind the construction. Any mistakes are the author's responsibility.

\section{A sufficient condition}

\subsection{A harmonic-measure criterion as an obstruction to cyclicity}

Here we present a condition on the size of $F$ that prevents cyclicity in $\B_0$. In fact, the criterion presented below gives more, since under the symmetry assumption
\[
F(-z)=\frac{1}{F(z)} \qquad z\in \D
\]
it further ensures that the invariant subspace generated by $F$ and $F^{-1}$ has index two. Recall that the harmonic measure of a Borel set
$E\subseteq\T$, viewed from $z\in\D$, is given by
\[
 \omega(z,\D,E)=\int_E P_z(\zeta)\,dm(\zeta),
\]
where $P_z$ denotes the Poisson kernel in $\D$. The criterion requires the large values of $F$ on selected circles to occupy sets of uniformly positive harmonic measure, as viewed from the preceding discs.
It is motivated by the use of growth as an obstruction to cyclicity in the work of Borichev and Hedenmalm \cite{borichevhedenmalmcyclicity}.

\begin{thm}[A harmonic-measure criterion]\label{THM:SUFFNONCYCLIC}
Let $F\in\B_0$ be zero-free and satisfy
\[
 F(-z)=F(z)^{-1},\qquad z\in\D.
\]
Let $\{r_j\}_{j\ge0}\subset(0,1)$ be strictly increasing to $1$,
and let $L_j>0$ satisfy
\[
 L_j\longrightarrow\infty,
 \qquad \varepsilon_j\sqrt{\log\frac e{1-r_j}}\lesssim L_j,
\]
for some numbers $\varepsilon_j>0$ with $\varepsilon_j\to0$ and $\varepsilon_{j-1}\lesssim\varepsilon_j$. If there exists $0<\gamma<1$ such that, for all sufficiently large $j$,
\begin{equation}\label{visible}
 \omega\left(\frac z{r_j},\D,
 \{\zeta\in\T:|F(r_j\zeta)|\ge L_j\}\right)\ge\gamma,
 \qquad |z|\le r_{j-1},
\end{equation}
then
\[
 \operatorname{ind}[F,F^{-1}]=2.
\]
In particular, neither $F$ nor $F^{-1}$ is cyclic in $\B_0$.
\end{thm}

\begin{remark}\label{REM:square}
For $F\in\B_0$, polynomial approximation and the square-mean
estimate \eqref{square} below give
\[
 \int_\T |F(r\zeta)|^2\,dm(\zeta)
 =o\!\left(\log\frac e{1-r}\right),\qquad r\to1^{-}.
\]
In order to obtain \eqref{visible}, we thus need $L_j$ to grow more
slowly than $\sqrt{\log(e/(1-r_j))}$, hence the role of $\varepsilon_j$, while still tending to infinity
to detect increasingly large values of $F$. This restriction is sharp, since along its radii $r_k$, the function $F$ constructed in Section~\ref{SEC:main} satisfies
\[
\int_\T|F(r_k\zeta)|^2\,dm(\zeta)\gtrsim\varepsilon_k^2\log\frac e{1-r_k},
 \qquad\varepsilon_k\to0,
\]
see \eqref{choice-scales} and \eqref{final-moments}. Its square means are thus as large as membership in $\B_0$ allows, up to the factor $\varepsilon_k^2$, while its large values occupy sets of uniformly positive harmonic measure. Since the amplitudes $\varepsilon_k$ in the construction may tend to zero arbitrarily slowly, this factor can be made as mild as one wishes. In this sense, $F$ realizes the Borichev--Hedenmalm principle within $\B_0$.
\end{remark}

\begin{remark}[A probabilistic interpretation]\label{REM:brownian}
Below we give a probabilistic interpretation of Theorem~\ref{THM:SUFFNONCYCLIC}, which goes back to Kakutani \cite{kakutani1944}. Since $\omega(z/r_j,\D,E)$ is the probability that a Brownian motion started at $z$ first leaves the disc $r_j\D$ through $r_jE$, condition \eqref{visible} says that a Brownian motion started in $r_{j-1}\D$ is stopped on $r_j\T$ in a \emph{cloud} with probability at least $\gamma$. Here a cloud is a part of $r_j\T$ where $|F|$ is at least of the order of its expected size, which is nearly maximal for functions in $\B_0$ by Remark~\ref{REM:square}. By symmetry, the clouds of $F^{-1}$ are the reflected ones.

Since this holds at every scale, a Brownian path from the origin avoids the clouds at $n$ consecutive scales with probability at most
\[
 (1-\gamma)^n\longrightarrow0,\qquad n\to\infty,
\]
so it is bound to be caught by a cloud sooner or later. This is the Borichev--Hedenmalm principle in $\B_0$, near-maximal growth on sets that Brownian motion cannot avoid, see Figure~\ref{FIG:visibility}. For an accessible introduction to the probabilistic perspective on harmonic measure, we refer to the book of {\O}ksendal \cite{oksendal2003}.
\end{remark}

\begin{figure}[!htbp]
\centering
\begin{minipage}[b]{0.49\textwidth}\centering
\includegraphics[width=\linewidth]{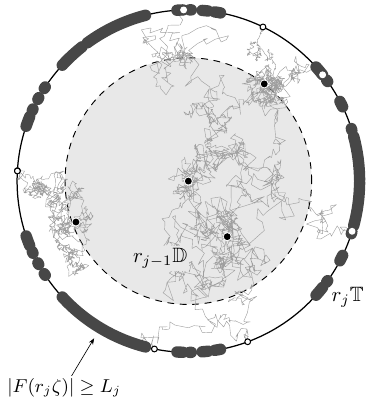}\\[2pt]
{\small (a) One scale}
\end{minipage}\hfill
\begin{minipage}[b]{0.46\textwidth}\centering
\includegraphics[width=\linewidth]{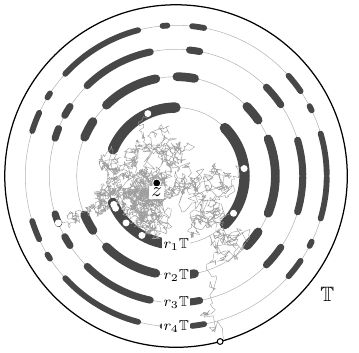}\\[2pt]
{\small (b) All scales}
\end{minipage}
\caption{Remark~\ref{REM:brownian} in pictures. (a) Brownian paths from points in $r_{j-1}\D$, stopped at $r_j\T$, and the dark "clouds" are where $|F|$ is large. (b) Each path is drawn until a cloud catches it. The clouds come from a simulated run of the construction.}\label{FIG:visibility}
\end{figure}

We shall use the following lemma, which we attribute to the works of Borichev \cite{borichev1998invariant}, and independently, Richter \cite{richter1987invariant}, which allows one to reformulate the matter of large index in terms of a norm estimate.

\begin{lemma}[Borichev--Richter lemma]\label{richter}
Let $f_1,\ldots,f_n\in\B_0$. Suppose that there exists $C>0$ such that
\begin{equation*}
 \sum_{j=1}^n|p_j(0)|
 \le C\norm{\sum_{j=1}^n p_jf_j}_\B
\end{equation*}
for all analytic polynomials $p_1,\ldots,p_n$. Then
\[
 \operatorname{ind}[f_1,\ldots,f_n]=n.
\]
\end{lemma}
For a self-contained short proof, see \cite[Lemma~3.1]{borichev1998invariant}.
\begin{proof}[Proof of Theorem~\ref{THM:SUFFNONCYCLIC}]
According to the Lemma of Borichev and Richter, it suffices to show that
\begin{equation}\label{richter-estimate}
 |p(0)|+|q(0)|\lesssim\norm{pF+qF^{-1}}_\B
\end{equation}
for analytic polynomials $p,q$. The idea is to compare their square means on successive circles $r_j \T$, moving outwards to $\T$, and then propagate the resulting bound to the origin.

\proofpart{1}{The square-mean estimate:}
For $g\in\B_0$, we first establish
\begin{equation}\label{square}
 \int_\T |g(r\zeta)|^2\,dm(\zeta)
 \lesssim\norm g_\B^2\log\frac e{1-r},
 \qquad 0<r<1.
\end{equation}
Indeed, the Littlewood--Paley identity and the Bloch estimate give
\[
 \begin{aligned}
 \int_\T |g(r\zeta)|^2\,dm(\zeta)
 &\lesssim |g(0)|^2+
 \int_{|\lambda|<r}|g'(\lambda)|^2
              \log\frac r{|\lambda|}\,dA(\lambda)\\
 &\lesssim\norm g_\B^2\left(1+
       \int_0^r\frac{t\log(r/t)}{(1-t^2)^2}\,dt\right)\\
 &\lesssim\norm g_\B^2\log\frac e{1-r}.
 \end{aligned}
\]
Here $dA$ denotes Lebesgue area measure, and the last inequality follows by integrating by parts.

\proofpart{2}{A harmonic-measure estimate for polynomial factors:}
Fix two analytic polynomials $p,q$ with
\[
 g(z)=p(z)F(z)+q(z)F^{-1}(z),\qquad \text{and} \quad \norm g_\B\le1,
\]
and set
\[
 \mathcal{I}_j^2:=\int_\T\bigl(|p(r_j\zeta)|^2+|q(r_j\zeta)|^2\bigr)
                    \,dm(\zeta), \qquad j\geq 1.
\]
Fix $|z|\le r_{j-1}$, write
$d\omega=P_{z/r_j}\,dm$, and note that by assumption the set
\[
 E:=\{\zeta\in\T:|F(r_j\zeta)|\ge L_j\}
\]
satisfies $\omega(E)\ge\gamma$ by \eqref{visible}.
Now on $E$, the identity $p=g/F-q/F^2$ gives 
\[
 |p(r_j\zeta)|
 \le L_j^{-1}|g(r_j\zeta)|+L_j^{-2}|q(r_j\zeta)|, \qquad \zeta \in E.
\]
Subharmonicity and Jensen's inequality on $E$ and its complement give
\[
 |p(z)|^2 \lesssim
 \left(\int_\T|p(r_j\zeta)|^2\,d\omega\right)^{1-\omega(E)} \cdot
 \left(\int_\T
 \bigl(L_j^{-2}|g(r_j\zeta)|^2+ L_j^{-4}|q(r_j\zeta)|^2\bigr)\,d\omega\right)^{\omega(E)}.
\]
Combining the estimate with
$|p(z)|^2\leq \int_\T|p(r_j\zeta)|^2\,d\omega$, with respective
powers $\gamma/\omega(E)$ and $1-\gamma/\omega(E)$, gives
\[
 \begin{aligned}
 |p(z)|^2
 &\lesssim
 \left(\int_\T|p(r_j\zeta)|^2\,d\omega\right)^{1-\gamma}\\
 &\quad\times
 \left(\int_\T
 \bigl(L_j^{-2}|g(r_j\zeta)|^2+
       L_j^{-4}|q(r_j\zeta)|^2\bigr)\,d\omega\right)^\gamma.
 \end{aligned}
\]
Average over $z=r_{j-1}\xi$. H\"older's inequality,
\eqref{square}, and the hypothesis on $L_j$ give
\[
 \begin{aligned}
 \int_\T|p(r_{j-1}\xi)|^2\,dm(\xi)
 &\lesssim
 \left(\int_\T|p(r_j\zeta)|^2\,dm(\zeta)\right)^{1-\gamma}\\
 &\quad\times
 \left(\frac1{L_j^2}\int_\T|g(r_j\zeta)|^2\,dm(\zeta)
       +\frac1{L_j^4}\int_\T|q(r_j\zeta)|^2\,dm(\zeta)\right)^\gamma\\
 &\lesssim \mathcal{I}_j^{2(1-\gamma)}
 \left(\varepsilon^{-2}_j+\frac{\mathcal{I}_j^2}{L_j^4}\right)^\gamma.
 \end{aligned}
\]
Now the symmetry of $F$ gives $|F^{-1}|\ge L_j$ on the reflected
set $-E$, and
\[
 \omega\left(\frac z{r_j},\D,-E\right)
 =\omega\left(-\frac z{r_j},\D,E\right)\ge\gamma.
\]
The same argument therefore applies to $q$. Adding the two
estimates and taking square roots yields
\begin{equation}\label{joint-recursion}
 \mathcal{I}_{j-1}\lesssim
 \varepsilon_j^{-\gamma}\mathcal{I}_j^{1-\gamma}+L_j^{-2\gamma}\mathcal{I}_j.
\end{equation}

\proofpart{3}{Propagating the estimate to the origin:}
Put $\alpha_j=\varepsilon_j^2\mathcal{I}_j$, and note that multiplying \eqref{joint-recursion}
by $\varepsilon_{j-1}^2$ gives
\[
 \alpha_{j-1}\le C\Bigl(\frac{\varepsilon_{j-1}}{\varepsilon_j}\Bigr)^2
 \left(\varepsilon_j^{\gamma}\alpha_j^{1-\gamma}+L_j^{-2\gamma}\alpha_j\right).
\]
Since $\varepsilon_j\to0$, $L_j\to\infty$ and $\varepsilon_{j-1}\lesssim\varepsilon_j$, both coefficients
tend to zero, so for a sufficiently large $j_0$, independent of $p$ and $q$, we have
\[
 \alpha_{j-1}\le \alpha_j^{1-\gamma}+\tfrac12\alpha_j,
 \qquad j\ge j_0+1.
\]
Since $p,q$ are polynomials, $S_0=\sup_{j\ge j_0}\alpha_j$ is finite, and thus taking the supremum gives
\[
 S_0\le S_0^{1-\gamma}+\tfrac12S_0,
 \qquad S_0\lesssim1.
\]
Consequently, we have arrived at the estimate
\[
 |p(0)|+|q(0)|\le2\mathcal{I}_{j_0}\le2\varepsilon_{j_0}^{-2}S_0\lesssim1,
\]
which upon restoring the normalization gives \eqref{richter-estimate}, hence completes the proof.
\end{proof}

\begin{remark}[The role of the symmetry]\label{REM:symmetry}
Apart from ensuring $F^{-1}\in\B_0$, the symmetry $F(-z)=F(z)^{-1}$ enters the proof only through the reflected sets $-E$. On these sets $F^{-1}$ is large, so that $q$ can be controlled exactly as $p$. It is this coupling of $p$ and $q$ that yields index two.

If one only asks for non-cyclicity of $F$, there is a separate criterion, based on subharmonicity and pointwise estimates instead of square means. The aim is then an estimate
\[
 \max_{|z|\le r_k}|p(z)|\le C_k\norm{pF}_\B
\]
for analytic polynomials $p$, where the constants $C_k$ grow sufficiently slowly. It is obtained by propagating the maxima of $|p|$ on the discs $r_j\D$ backwards to the origin, as for the square means in Step~3, but it requires additional conditions on the radii $r_j$. In the limit, one gets $|h|\le C_k|F|\norm h_\B$ on $r_k\D$ for every $h\in[F]$. Hence $1\notin[F]$, provided that $F$ is sufficiently small, compared with $C_k^{-1}$, somewhere on one of the discs $r_k\D$. We do not carry out the details here.
\end{remark}

\subsection{A general principle for preventing cyclicity}\label{SEC:general}
The proof of Theorem~\ref{THM:SUFFNONCYCLIC} uses the space $\B_0$ only through the square-mean estimate \eqref{square} and the Borichev--Richter lemma. It therefore extends to a broad class of spaces of analytic functions, as follows.

\begin{thm}[A general harmonic-measure criterion]\label{THM:general}
Let $X$ be a Banach space of analytic functions in $\D$ with the following two properties.
\begin{enumerate}
 \item[(a)] Point evaluations are continuous on $X$, the analytic polynomials are dense in $X$, and $M_z$ is bounded and bounded below on $X$.
 \item[(b)] There exists a function $\Phi$ with $\Phi(r)\to\infty$ as $r\to1^-$, such that
 \[
  \int_\T|g(r\zeta)|^2\,dm(\zeta)\lesssim\|g\|_X^2\,\Phi(r),\qquad g\in X,\quad 0<r<1.
 \]
\end{enumerate}
Let $F$ be a zero-free analytic function in $\D$ with $F,F^{-1}\in X$ and $F(-z)=F(z)^{-1}$. Let $\{r_j\}_{j\ge0}\subset(0,1)$ be strictly increasing to $1$, and let $L_j>0$ satisfy
\[
 L_j\longrightarrow\infty,
 \qquad \varepsilon_j\sqrt{\Phi(r_j)}\lesssim L_j,
\]
for some numbers $\varepsilon_j>0$ with $\varepsilon_j\to0$ and $\varepsilon_{j-1}\lesssim\varepsilon_j$. If there exists $0<\gamma<1$ such that \eqref{visible} holds for all sufficiently large $j$, then $[F,F^{-1}]$ has index $2$ in $X$. In particular, neither $F$ nor $F^{-1}$ is cyclic in $X$.
\end{thm}

\begin{proof}[Sketch of proof]
The proof of Theorem~\ref{THM:SUFFNONCYCLIC} carries over with two changes.
\begin{enumerate}
 \item[(1)] By (a), the Borichev--Richter lemma applies in $X$, so it suffices to show that
 \[
  |p(0)|+|q(0)|\lesssim\|pF+qF^{-1}\|_X
 \]
 for all analytic polynomials $p$ and $q$.
 \item[(2)] Assumption (b) takes the place of \eqref{square}, with $\Phi(r)$ in place of $\log\frac e{1-r}$. Together with the hypothesis $\varepsilon_j\sqrt{\Phi(r_j)}\lesssim L_j$, it gives the bound in Step~2.
\end{enumerate}
The rest of the argument involves only harmonic measure, the symmetry of $F$ and the polynomials $p$ and $q$.
\end{proof}

Examples include the Bergman spaces $A^p$ with $1\le p<\infty$, where $\Phi(r)=(1-r)^{-\beta}$ for a suitable $\beta=\beta_p>0$, and the weighted little Bloch spaces discussed after Corollary~\ref{COR:CYCINV}, with $\Phi=\Phi_w$.

\medskip\paragraph{\textbf{A new manifestation of the Borichev--Hedenmalm principle.}}
In the work of Borichev and Hedenmalm \cite{borichev1995cyclicity,borichevhedenmalmcyclicity}, the impediment to cyclicity comes from the geometry of the sets where the function is large. Their ``mushroom forest'' construction is designed so that the unit circle has harmonic measure zero in the complement of the forest, as seen from any point inside the disc, and it is this property which allows the polynomial factors to be controlled.

Theorem~\ref{THM:general} rests on the same mechanism, but asks for much less, namely that the sets of large values are visible in harmonic measure, uniformly at every scale. No explicit geometry of these sets is needed, which is what allows the visibility to be obtained from moment estimates alone in Section~3. Moreover, the criterion identifies $[F,F^{-1}]$ as an invariant subspace of index two.

\section{Constructing the functions \texorpdfstring{$F$}{F}}

\subsection{Outline of the proof strategy}\label{SEC:GUIDE}
The idea is to construct $F$ of the form
\[
F(z) = \exp f(z), \qquad z\in \D
\]
by successively approximating $f$ with odd analytic
polynomials $f_k$, chosen so that
\begin{enumerate}
    \item[(i)] $F \in \B_0$,
    \item[(ii)] $F$ satisfies the condition \eqref{visible} in Theorem~\ref{THM:SUFFNONCYCLIC}.
\end{enumerate}
Indeed, once $F$ belongs to $\B_0$, the oddness of each $f_k$ forces 
\[
\frac{1}{F(z)} = F(-z)  \in \B_0,
\]
while \eqref{visible} prevents $F$ from being cyclic in $\B_0$. 

To carry out this program, we are therefore faced with the following sub-problems:

\begin{enumerate}
    \item[1.)] Construct the functions $f_k$ in such a way that $F \in \B_0$.
    \item[2.)] Obtain the uniform lower bound on harmonic measure in \eqref{visible}.
\end{enumerate}
Below, we explain these two components of the construction.

\paragraph{\textbf{Accomplishing membership in $\B_0$.}}
The polynomials $f_k$ will be constructed recursively by
adding polynomial terms of the form
\begin{equation}\label{DEF:fk}
 f_k(z)=f_{k-1}(z)+\sum_{j=1}^{M_k}z^{N_j} P_j(z), \qquad N_j>0
\end{equation}
The principal difficulty is to keep the accumulated changes in
$e^{\pm f_k}$ small in Bloch norm, while simultaneously allowing a sufficiently long block of $M_k$ terms to produce the large values
required in the second part of the construction.

At each insertion, the \emph{damping polynomial} $P_j$ compensates
for the size of the preceding exponential. If the preceding exponent
is $f_{k-1}$, this means
\[
 |P_j(\zeta)|e^{|\Re f_{k-1}(\zeta)|}
 \asymp\varepsilon_k,\qquad \zeta\in\T.
\]
Meanwhile, multiplication by $z^{N_j}$ makes each added term vanish
to high order, and therefore places its change near the boundary.
Indeed, Lemma~\ref{frequency-selection} gives
\[
 (1-|z|)
 \bigl|(e^{f_{k-1}+z^{N_j}P_j}-e^{f_{k-1}})'(z)\bigr|
 \lesssim\varepsilon_k N_j(1-|z|)|z|^{N_j-1},
\]
and its right-hand side is concentrated where
$1-|z|\asymp N_j^{-1}$. Since the frequencies $N_j$ grow geometrically, these regions are well separated, and at each point $z$ only boundedly many of the terms are not geometrically small. This allows us to add $M_k$ terms of the form $z^{N_j}P_j(z)$, while keeping the accumulated Bloch norm small.

\paragraph{\textbf{Accomplishing a uniform lower bound on harmonic measure.}}

We shall use the block length $M_k$ in defining $f_k$ in \eqref{DEF:fk}, to produce large values of $e^{f_k}$. The difficulty is
to ensure that these values occupy a set of uniformly
positive harmonic measure, rather than increasingly
small exceptional sets. The construction will be guided by the following lemma from probability theory.

\begin{lemma}[Paley--Zygmund]\label{paley-zygmund}
If $X$ is a non-negative random variable on a probability space with $0<\mathbb{E}(X^2)<\infty$, then
\[
\mathbb{P}\left(X\geq \frac{1}{2} \,\mathbb{E}X\right)
\geq \frac{1}{4}\frac{(\mathbb{E}X)^2}{\mathbb{E}(X^2)}.
\]
\end{lemma}
\begin{proof}
Since the proof is short and simple, we provide it for the reader's convenience. Put $E=\{X\geq \frac{1}{2} \,\mathbb{E}X\}$, and note that the Cauchy--Schwarz inequality gives
\[
\mathbb{E}X
\leq \frac{1}{2} \mathbb{E}X+\mathbb{E}(X\mathbf{1}_E)
\leq \frac{1}{2} \mathbb{E}X+
\bigl(\mathbb{E}(X^2)\bigr)^{1/2}\, \mathbb{P}(E)^{1/2}.
\]
The claim follows from a re-arrangement.
\end{proof}
For $|z|\le r_{k-1}$, consider the random variable
$X(\zeta)=|F(r_k\zeta)|^2$ on $\T$ with probability measure
\[
 d\mathbb P_k(\zeta)=P_{z/r_k}(\zeta)\,dm(\zeta).
\]
We shall choose thresholds $L_k\asymp\varepsilon_k\sqrt{M_k}$
and establish
\begin{equation}\label{EQ:2nd4th}
 \begin{aligned}
 \mathbb E X
 &=\int_\T |F(r_k\zeta)|^2P_{z/r_k}(\zeta)\,dm(\zeta)
 \gtrsim L_k^2,\\
 \mathbb E(X^2)
 &=\int_\T |F(r_k\zeta)|^4P_{z/r_k}(\zeta)\,dm(\zeta)
 \lesssim L_k^4.
 \end{aligned}
\end{equation}
This is where the block length $M_k$ in \eqref{DEF:fk} enters. Each insertion in the $k$th block raises the second moment by roughly $\varepsilon_k^2$, so a block of length $M_k$ produces $\mathbb EX\asymp\varepsilon_k^2M_k$, while the damping keeps $\mathbb E(X^2)$ of the order $(\varepsilon_k^2M_k)^2$. Proposition~\ref{PROP:blocks} below provides these estimates for $e^{f_k}$. They persist for $F$, since all later terms vanish to such high order that they are negligible on the disc $r_k\D$.

The lower bound on $\mathbb E X$ guarantees large values,
while the upper bound on $\mathbb E(X^2)$ prevents these
values from being confined to sets of arbitrarily small
harmonic measure. Choosing the constant in
$L_k\asymp\varepsilon_k\sqrt{M_k}$ sufficiently small,
the Paley--Zygmund lemma gives
\[
 \mathbb{P}_k\left(
 \left\{\zeta\in\T:
 |F(r_k\zeta)|\ge L_k\right\}\right)
 \gtrsim1,
\]
which is the required lower bound in \eqref{visible}.
It remains to choose the block lengths and radii so that
\[
 L_k\longrightarrow\infty,
 \qquad
 \varepsilon_k\sqrt{\log\frac e{1-r_k}}\lesssim L_k.
\]
Both are achieved by the fixed frequency ratio within each block, that is, $N_j=\kappa^jN_0$. The $k$th block then ends at the frequency $\kappa^{M_k}N_0$, and the radius $r_k$ is chosen at this scale, so that
\[
 \log\frac e{1-r_k}\asymp M_k,
\]
see \eqref{choice-scales}. Hence $\varepsilon_k\sqrt{\log\frac e{1-r_k}}$ is comparable to $\varepsilon_k\sqrt{M_k}\asymp L_k$, and $L_k\to\infty$ as soon as the block lengths $M_k$ are taken sufficiently large.

This is where the difference from the ``mushroom forest'' construction of Borichev and Hedenmalm, discussed in Section~\ref{SEC:general}, becomes concrete. The Paley--Zygmund lemma produces the visibility from the two moment bounds in \eqref{EQ:2nd4th} alone, without any information on the geometry of the sets where $F$ is large, and this flexibility is what we exploit in Proposition~\ref{PROP:blocks}.

\subsection{The principal building blocks}
Each added term is built from two ingredients. A damping polynomial compensates for the size of the preceding exponential, and a factor $z^N$ places the change near the boundary, at the scale $1/N$. We first construct the damping polynomials.

\begin{lemma}[Damping Lemma]\label{LEM:damping}
Let $0<\varepsilon<1$ and let $Q$ be an analytic polynomial. There exists an analytic polynomial $P$ without zeros in $\cD$ such that
\begin{equation}\label{damping}
 c\varepsilon\le e^{|\Re Q|}|P|\le\varepsilon\qquad\text{on }\T,
\end{equation}
where $c>0$ is an absolute constant. If $Q$ is odd, then $P$ can be chosen even.
\end{lemma}

\begin{proof}
Since $\Re Q$ is a trigonometric polynomial, the weight $e^{-2|\Re Q|}$ is positive and continuous on $\T$, and a Fej\'er mean $T$ of it of sufficiently high order satisfies
\[
 \tfrac12e^{-2|\Re Q|}\le T\le\tfrac32e^{-2|\Re Q|}\qquad\text{on }\T.
\]
As $T$ is a strictly positive trigonometric polynomial, the Fej\'er--Riesz theorem provides an analytic polynomial $p$ without zeros in $\cD$ such that $T=|p|^2$ on $\T$. Taking $P=\sqrt{2/3}\,\varepsilon p$, we obtain
\[
 \frac{\varepsilon}{\sqrt3}\le e^{|\Re Q|}|P|\le\varepsilon\qquad\text{on }\T,
\]
which is \eqref{damping} with $c=1/\sqrt3$. Such a polynomial $P$ is determined by its modulus on $\T$ up to a unimodular constant, which we may fix by $P(0)>0$. If $Q$ is odd, then the weight and its Fej\'er means are even, so $|P|$ is even on $\T$, and hence $P(-z)=P(z)$.
\end{proof}

We now insert the damping polynomials at lacunary frequencies, prescribed in advance.

\begin{lemma}\label{frequency-selection}
Fix an analytic polynomial $Q_0$ and amplitudes
$0<\varepsilon_j\le1/100$. For a sufficiently large
integer $\kappa>0$, and a sufficiently large initial frequency $N_0>0$, the
recursion
\[
 Q_j=Q_{j-1}+z^{N_j}P_j,\qquad N_j\ge \kappa N_{j-1},
\]
can be carried out with
\[
 c\varepsilon_j\le |P_j(\zeta)|e^{|\Re Q_{j-1}(\zeta)|}
 \le\varepsilon_j,\qquad \zeta\in\T.
\]
The exponential increments satisfy
\begin{equation}\label{tail}
 |e^{\pm Q_j(z)}-e^{\pm Q_{j-1}(z)}|\lesssim\varepsilon_j|z|^{N_j},
 \qquad
 |(e^{\pm Q_j}-e^{\pm Q_{j-1}})'(z)|
 \lesssim\varepsilon_jN_j|z|^{N_j-1},\qquad z\in\D.
\end{equation} If $Q_0$ and all $N_j$ are odd, the $P_j$ can be chosen even, so that every $Q_j$ is odd.
\end{lemma}

\begin{proof} 
\proofpart{0}{Strategy and notions:}
Since $N_j$ is prescribed, we need control of $P_j$ at the scale $1/N_j$, which we shall obtain by constructing $P_j$ on the slightly larger circle $R_j\T$. Cauchy estimates then converts boundedness into the required derivative bounds.
To this end, we shall write 
\[
\|h\|_{L^\infty(R\D)}=\max_{|z|\le R}|h(z)|, \qquad R>0
\]
and put $R_j=1+1/N_j$. The point of the proof is that the damping survives the passage from $R_j\T$ to $\T$.

\proofpart{1}{Construction of $P_j$ and the upper bounds:}
Applying Lemma~\ref{LEM:damping} to $Q_{j-1}(R_jz)$ and rescaling, we obtain a polynomial $P_j$ without zeros in $R_j\cD$ with
\begin{equation}\label{outer-circle-damping}
 c\varepsilon_j\le|P_j(R_j\eta)|\,e^{|\Re Q_{j-1}(R_j\eta)|}\le\varepsilon_j,\qquad\eta\in\T.
\end{equation}
 The maximum principle, applied to $P_je^{Q_{j-1}}$ and $P_je^{-Q_{j-1}}$, extends the upper bound in \eqref{outer-circle-damping} to $|z|\le R_j$. In particular $\|z^{N_j}P_j\|_{L^\infty(R_j\D)}\lesssim\varepsilon_j$, since $R_j^{N_j}\le e$. Writing
\[
 D_j=e^{Q_j}-e^{Q_{j-1}}=e^{Q_{j-1}}\bigl(e^{z^{N_j}P_j}-1\bigr)
\]
and using $|e^w-1|\le|w|e^{|w|}$, we get $\|D_j\|_{L^\infty(R_j\D)}\lesssim\varepsilon_j$. As $D_j$ vanishes to order $N_j$ at the origin, the Schwarz lemma and Cauchy's estimate give
\[
 |D_j(z)|\lesssim\varepsilon_j|z|^{N_j},\qquad |D_j'(z)|\lesssim\varepsilon_jN_j|z|^{N_j-1},\qquad z\in\D.
\]
The same argument applies to $-Q_j$. If $Q_{j-1}$ is odd, then $P_j$ is even by Lemma~\ref{LEM:damping}, so $z^{N_j}P_j$ is odd for odd $N_j$.

\proofpart{2}{An oscillation estimate:}
For an analytic polynomial $h$, radii $1<R<R'$ and $\zeta\in\T$, we have
\begin{equation}\label{poisson-comparison}
 \int_\T|h(R\eta)-h(\zeta)|^2P_{\zeta/R}(\eta)\,dm(\eta)\lesssim\frac{R^2-1}{R'^2-1}\,\|h\|_{L^\infty(R'\D)}^2.
\end{equation}
Indeed, with $g(z)=(h(z)-h(\zeta))/(z-\zeta)$ and $P_{\zeta/R}(\eta)=(R^2-1)|R\eta-\zeta|^{-2}$, the monotonicity of the square means of $g$ gives
\[
 \begin{aligned}
 \int_\T|h(R\eta)-h(\zeta)|^2P_{\zeta/R}(\eta)\,dm(\eta)
 &=(R^2-1)\int_\T|g(R\eta)|^2\,dm(\eta)
 \le(R^2-1)\int_\T|g(R'\eta)|^2\,dm(\eta)\\
 &\le4\|h\|_{L^\infty(R'\D)}^2\,(R^2-1)\int_\T\frac{dm(\eta)}{|R'\eta-\zeta|^2}
 =4\,\frac{R^2-1}{R'^2-1}\,\|h\|_{L^\infty(R'\D)}^2.
 \end{aligned}
\]
We apply \eqref{poisson-comparison} with $R=R_j$ to $Q_0$, with $R'=2$, and to each term $z^{N_\ell}P_\ell$ of $Q_{j-1}$, with $R'=R_\ell$. Since $R_j^2-1\asymp1/N_j$ and $\|z^{N_\ell}P_\ell\|_{L^\infty(R_\ell\D)}\lesssim\varepsilon_\ell$, Minkowski's inequality gives
\begin{equation}\label{damping-oscillation}
 \Bigl(\int_\T|Q_{j-1}(R_j\eta)-Q_{j-1}(\zeta)|^2P_{\zeta/R_j}(\eta)\,dm(\eta)\Bigr)^{1/2}
 \lesssim\frac{\|Q_0\|_{L^\infty(2\D)}}{\sqrt{N_j}}+\sum_{\ell<j}\varepsilon_\ell\sqrt{\frac{N_\ell}{N_j}}\lesssim1,
\end{equation}
where the last step uses $N_\ell/N_j\le\kappa^{\ell-j}$ and $N_0\ge\|Q_0\|_{L^\infty(2\D)}^2$.

\proofpart{3}{The lower damping bound on $\T$:}
Fix $\zeta\in\T$. Since $P_j$ has no zeros in $R_j\cD$, the function $\log|P_j|$ is harmonic there. By the Poisson formula on $R_j\D$ and the lower bound in \eqref{outer-circle-damping},
\[
 \log|P_j(\zeta)|\ge\log(c\varepsilon_j)-\int_\T|\Re Q_{j-1}(R_j\eta)|\,P_{\zeta/R_j}(\eta)\,dm(\eta).
\]
Since $|\Re Q_{j-1}(R_j\eta)|\le|\Re Q_{j-1}(\zeta)|+|Q_{j-1}(R_j\eta)-Q_{j-1}(\zeta)|$, the Cauchy--Schwarz inequality and \eqref{damping-oscillation} bound the integral by $|\Re Q_{j-1}(\zeta)|+C$. Hence $|P_j(\zeta)|\,e^{|\Re Q_{j-1}(\zeta)|}\gtrsim\varepsilon_j$, which is the lower damping bound.

\end{proof}

\subsection{Assembling the building blocks}
We now iterate Lemma~\ref{frequency-selection} in blocks. Starting from $Q_0=f_{k-1}$, we add $M_k$ terms of amplitude $\varepsilon_k$, choosing each $P_j$ to damp the exponentials obtained at the preceding insertion,
\[
 Q_j=Q_{j-1}+z^{N_j}P_j,\qquad 1\le j\le M_k,
 \qquad f_k=Q_{M_k}.
\]
Increasing $M_k$ will build up the second moment, while damping
keeps the fourth moment under control.

\begin{prop}\thlabel{PROP:blocks}
There is an absolute odd integer $\kappa>1$ with the following property.
Fix an analytic polynomial $Q_0$, an amplitude
$0<\varepsilon_k\le1/100$, and $0\le\rho<1$.
For every sufficiently large initial frequency $N_0$ (depending on $Q_0$, $\varepsilon_k$ and $\rho$), we can construct analytic polynomials
\[
 Q_{M_k}(z)=Q_0(z)+\sum_{j=1}^{M_k}z^{N_j}P_j(z),
 \qquad N_j=\kappa^jN_0,
\]
with damping polynomials $P_j$ as in Lemma~\ref{frequency-selection}, with amplitude $\varepsilon_k$.
For all sufficiently large block lengths $M_k$ (depending also on $N_0$), we have
\begin{equation}\label{block-proposition-moments}
 \int_\T e^{2\Re Q_{M_k}}P_a\,dm\gtrsim M_k\varepsilon_k^2,
 \qquad
 \int_\T e^{4\Re Q_{M_k}}P_a\,dm\lesssim (M_k\varepsilon_k^2)^2,
 \qquad |a|\le\rho.
\end{equation}
These estimates hold uniformly for $|a|\le\rho$.
\end{prop}

We now compare the second and fourth moments of $e^{Q_j}$ with those of $e^{Q_{j-1}}$. The following lemma is the analytic form of the orthogonality of martingale increments.

The increment $D_j$ lives at frequencies of order at least $N_j$, whereas $e^{Q_{j-1}}$ varies on scales up to $N_{j-1}=N_j/\kappa$, so the two are almost orthogonal. Up to errors of relative size $1/\kappa$, the second moment therefore grows by $\mathbb E|D_j|^2$, and the fourth moment grows by at most a multiple of $\varepsilon_k^2\,\mathbb EX_{j-1}$, as for a martingale with increments of size $\varepsilon_k$.

\begin{lemma}\label{high-frequency-averaging}
Consider the blocks in Lemma~\ref{frequency-selection} with common
amplitude $\varepsilon_k$ and $N_j=\kappa^jN_0$. Put
\[
 X_j=e^{2\Re Q_j},\qquad D_j=e^{Q_j}-e^{Q_{j-1}},
\]
and take expectations with respect to $d\mathbb P=P_a\,dm$.
For each fixed $\rho<1$, choosing $N_0$ sufficiently large gives
\begin{equation}\label{EQ:oscillatory-terms}
 \mathbb E X_j
 =\mathbb E X_{j-1}+\mathbb E|D_j|^2+O(\varepsilon_k^2/\kappa),
\end{equation}
and
\begin{equation}\label{EQ:fourth-increment}
 \mathbb E(X_j^2)
 \le\mathbb E(X_{j-1}^2)+C\varepsilon_k^2\mathbb E X_{j-1},
\end{equation}
uniformly for $|a|\le\rho$ and all $j\ge1$. The same estimates hold with $Q_j$ replaced by $-Q_j$.
\end{lemma}
\begin{proof}
Since $e^{Q_j}=e^{Q_{j-1}}+D_j$, we have on $\T$
\[
 \begin{aligned}
 X_j-X_{j-1}&=|D_j|^2+2\Re\bigl(\overline{e^{Q_{j-1}}}D_j\bigr),\\
 X_j^2-X_{j-1}^2&=4X_{j-1}\Re\bigl(\overline{e^{Q_{j-1}}}D_j\bigr)+O\bigl(\varepsilon_k^2X_{j-1}\bigr),
 \end{aligned}
\]
where the second identity follows by expanding $|e^{Q_{j-1}}+D_j|^4$ and using $|D_j|\le|e^{Q_{j-1}}|$ and $|D_j|\lesssim\varepsilon_k$ on $\T$. Both claims therefore reduce to estimates of the mixed terms, namely
\begin{equation}\label{EQ:mixed}
 \bigl|\mathbb E\bigl(\overline{e^{Q_{j-1}}}D_j\bigr)\bigr|\lesssim\frac{\varepsilon_k^2}{\kappa},
 \qquad
 \bigl|\mathbb E\bigl(X_{j-1}\overline{e^{Q_{j-1}}}D_j\bigr)\bigr|\lesssim\frac{\varepsilon_k^2}{\kappa}\,\mathbb EX_{j-1}.
\end{equation}

These terms are small because $D_j$ oscillates much faster than $e^{Q_{j-1}}$, and we exploit this by integrating by parts in the angular variable. The function
\[
 \Delta_j(z)=\int_0^1D_j(rz)\,\frac{dr}{r}
\]
is analytic in $\D$ and satisfies the identity $\partial_\theta\Delta_j(e^{i\theta})=iD_j(e^{i\theta})$. Since $D_j$ vanishes to order $N_j$ at the origin, the Schwarz lemma gives $|\Delta_j|\lesssim\varepsilon_k\int_0^1r^{N_j-1}\,dr=\varepsilon_k/N_j$. Integrating by parts, we obtain
\[
 \mathbb E\bigl(\overline{e^{Q_{j-1}}}D_j\bigr)
 =i\int_\T\Delta_j\,\partial_\theta\bigl(\overline{e^{Q_{j-1}}}P_a\bigr)\,dm.
\]
Since $|\partial_\theta P_a|\le2P_a/(1-\rho)$ for $|a|\le\rho$, this yields
\[
 \bigl|\mathbb E\bigl(\overline{e^{Q_{j-1}}}D_j\bigr)\bigr|
 \lesssim\frac{\varepsilon_k}{N_j}\Bigl(\bigl\|(e^{Q_{j-1}})'\bigr\|_\infty+\frac{\|e^{Q_{j-1}}\|_\infty}{1-\rho}\Bigr).
\]
It thus remains to estimate norms inside the bracket. Summing the bounds of Lemma~\ref{frequency-selection} over the earlier insertions gives
\[
 \bigl\|(e^{Q_{j-1}})'\bigr\|_\infty\lesssim\bigl\|(e^{Q_0})'\bigr\|_\infty+\varepsilon_kN_{j-1},
 \qquad
 \|e^{Q_{j-1}}\|_\infty\lesssim\|e^{Q_0}\|_\infty+j\varepsilon_k.
\]
Since $N_{j-1}=\kappa^{j-1}N_0$ grows geometrically in $j$, the terms involving $Q_0$, $\rho$ and $j$ are at most $\varepsilon_kN_{j-1}$ once $N_0$ is sufficiently large. The bracket is therefore $\lesssim\varepsilon_kN_{j-1}$, and the first estimate in \eqref{EQ:mixed} follows from $N_{j-1}/N_j=1/\kappa$. The same argument gives
\begin{equation}\label{EQ:block-derivative}
 \norm{(e^{\pm Q_j})'}_\infty\lesssim\varepsilon_kN_j.
\end{equation}
For the second estimate in \eqref{EQ:mixed}, we integrate by parts in the same way, now with the additional factor $X_{j-1}=|e^{Q_{j-1}}|^2$. Its angular derivative is at most $2|e^{Q_{j-1}}|\,|(e^{Q_{j-1}})'|$, so the same bound holds with the additional factor $\mathbb EX_{j-1}$ on the right-hand side.
\end{proof}
Now turn to the proof of the principal result in this subsection.
\begin{proof}[Proof of Proposition~\ref{PROP:blocks}]
Choose $N_0$ large enough for Lemma~\ref{high-frequency-averaging}, and consider the blocks
in Lemma~\ref{frequency-selection} with amplitude $\varepsilon_k$
and $N_j=\kappa^jN_0$. All estimates are uniform for $|a|\le\rho$.

\proofpart{1}{The effect of a single insertion:}
By Lemma~\ref{high-frequency-averaging},
\[
 \int_\T e^{2\Re Q_j}P_a\,dm
 -\int_\T e^{2\Re Q_{j-1}}P_a\,dm
 =\int_\T|D_j|^2P_a\,dm+O(\varepsilon_k^2/\kappa).
\]
To estimate the integral on the right, use
$|P_j|\le\varepsilon_k\le1/100$ and the damping property:
\[
 |D_j|^2=e^{2\Re Q_{j-1}}\bigl|e^{\zeta^{N_j}P_j}-1\bigr|^2\asymp e^{2\Re Q_{j-1}}|P_j|^2\asymp\varepsilon_k^2e^{2\Re Q_{j-1}-2|\Re Q_{j-1}|},\qquad\zeta\in\T.
\]
Substituting this comparison into the preceding identity, and
applying the same argument to $-Q_j$, we obtain
\begin{equation}\label{EQ:smalldiff}
 \begin{aligned}
 \int_\T\bigl(e^{\pm2\Re Q_j}-e^{\pm2\Re Q_{j-1}}\bigr)P_a\,dm&\le C\varepsilon_k^2,\\
 \int_\T\bigl(e^{\pm2\Re Q_j}-e^{\pm2\Re Q_{j-1}}\bigr)P_a\,dm&\ge\varepsilon_k^2\Bigl(c\int_\T e^{\pm2\Re Q_{j-1}-2|\Re Q_{j-1}|}P_a\,dm-\frac C\kappa\Bigr).
 \end{aligned}
\end{equation}
Here the upper bound uses $e^{\pm2u-2|u|}\le1$.
The fourth-moment estimate is \eqref{EQ:fourth-increment}.

\proofpart{2}{The upper moment estimates:}
Put $A_j=j\varepsilon_k^2$. For the second moments, iterating
\eqref{EQ:smalldiff} gives
\[
 \int_\T e^{\pm2\Re Q_j}P_a\,dm
 \le\int_\T e^{\pm2\Re Q_0}P_a\,dm+CA_j.
\]
For the fourth moments, we iterate \eqref{EQ:fourth-increment}
and use the preceding bound:
\[
 \int_\T e^{\pm4\Re Q_j}P_a\,dm
 \le\int_\T e^{\pm4\Re Q_0}P_a\,dm
       +CA_j\int_\T e^{\pm2\Re Q_0}P_a\,dm+CA_j^2.
\]
The initial polynomial is fixed. Thus, for all sufficiently large $j$,
\begin{equation}\label{EQ:mom2est}
 \int_\T e^{\pm2\Re Q_j}P_a\,dm\lesssim A_j,
 \qquad
 \int_\T e^{\pm4\Re Q_j}P_a\,dm\lesssim A_j^2,
\end{equation}
with absolute comparison constants.

\proofpart{3}{Growth of the logarithmic mean:}
This step shows that $\Re Q_j$ spreads out, in the sense that $|\Re Q_j|$ is typically of size $\log A_j$. Since
\[
 e^{2u-2|u|}+e^{-2u-2|u|}=1+e^{-4|u|}\ge1,
 \qquad u\in\mathbb R,
\]
the two lower bounds in \eqref{EQ:smalldiff}, with
$X_j=e^{2\Re Q_j}$ as in Lemma~\ref{high-frequency-averaging}, give
\[
 \mathbb E(X_j+X_j^{-1})-\mathbb E(X_{j-1}+X_{j-1}^{-1})
 \ge\left(c-\frac C\kappa\right)\varepsilon_k^2
 \gtrsim\varepsilon_k^2,
\]
where the last inequality follows by choosing $\kappa$ sufficiently large.
Summing from the first insertion to the $j$th, and using
$A_j=j\varepsilon_k^2$, yields
\[
 \int_\T\bigl(e^{2\Re Q_j}+e^{-2\Re Q_j}\bigr)P_a\,dm
 \ge\mathbb E(X_0+X_0^{-1})+cj\varepsilon_k^2
 \gtrsim A_j.
\]
Applying Paley--Zygmund as in Subsection~\ref{SEC:GUIDE}, the
fourth-moment bound in \eqref{EQ:mom2est} therefore gives
\[
 \int_{\{|\Re Q_j|\ge c\log A_j\}}P_a\,dm\gtrsim1
\]
for all sufficiently large $j$. Here we used
$e^{2u}+e^{-2u}\le2e^{2|u|}$. Consequently,
\begin{equation}\label{logarithmic-mean}
 \int_\T|\Re Q_j|P_a\,dm\gtrsim\log A_j.
\end{equation}

\proofpart{4}{Positive real parts on a set of positive harmonic measure:}
The point of this step is that $\Re Q_j$ cannot be large only in absolute value. Its mean against $P_a\,dm$ is the bounded number $\Re Q_j(a)$, so large negative values must be balanced by large positive ones, and the bound on $e^{2\Re Q_j}$ in \eqref{EQ:mom2est} prevents the positive values from concentrating on a small set. The signed mean remains bounded independently of the block length:
\begin{equation}\label{EQ:block-interior}
 \begin{aligned}
 \left|\int_\T\Re Q_j\,P_a\,dm\right|
 &=|\Re Q_j(a)|\\
 &\le\norm{\Re Q_0}_\infty
       +\varepsilon_k\sum_{i=1}^j\rho^{N_i}
 \le\norm{\Re Q_0}_\infty+\frac{\varepsilon_k}{1-\rho}.
 \end{aligned}
\end{equation}
Combining this with \eqref{logarithmic-mean} gives, for all sufficiently large $j$,
\[
 2\int_\T(\Re Q_j)_+P_a\,dm
 =\int_\T|\Re Q_j|P_a\,dm+\Re Q_j(a)
 \gtrsim\log A_j.
\]
Set $E_j=\{\Re Q_j>0\}$ and $\omega_j=\int_{E_j}P_a\,dm$.
Jensen's inequality and \eqref{EQ:mom2est} give
\[
 \begin{aligned}
 2\int_\T(\Re Q_j)_+P_a\,dm
 &\le\omega_j\log\left(
       \frac1{\omega_j}\int_{E_j}e^{2\Re Q_j}P_a\,dm\right)\\
 &\le\omega_j\log(CA_j)+\frac1e.
 \end{aligned}
\]
Comparing the two bounds yields $\omega_j\gtrsim1$. More precisely, $\omega_j\ge\delta_0$ for all sufficiently large $j$, where $\delta_0>0$ depends only on the absolute constants in \eqref{EQ:mom2est} and \eqref{logarithmic-mean}. These constants do not depend on $\kappa$, once $\kappa$ exceeds the threshold $\kappa_1$ required in Step~3.
Since $e^{2\Re Q_j-2|\Re Q_j|}=1$ on $E_j$,
\eqref{EQ:smalldiff} now gives
\[
 \int_\T\bigl(e^{2\Re Q_{j+1}}-e^{2\Re Q_j}\bigr)P_a\,dm
 \ge c\varepsilon_k^2-\frac{C\varepsilon_k^2}{\kappa}
 \ge\frac c2\varepsilon_k^2,
\]
where $c$ now contains the factor $\delta_0$. Since $\delta_0$ does not depend on $\kappa\ge\kappa_1$, we may now fix $\kappa\ge\kappa_1$, large enough for this inequality and for Lemma~\ref{frequency-selection}.
Summing over the last half of a sufficiently long block proves
the lower bound in \eqref{block-proposition-moments}.
The upper bound follows from Step~2.
\end{proof}

\section{Proof of the main result}\label{SEC:main}

\subsection{The order of the parameters}
We construct odd polynomials $f_k$ and radii $r_k\uparrow1$ inductively.
The frequency ratio $\kappa$ is the absolute odd integer of Proposition~\ref{PROP:blocks}, fixed once and for all.
The amplitudes $\varepsilon_k$, which control the exponential increments, are also fixed in advance. We only use that they decrease to zero and satisfy the doubling condition $\varepsilon_{k-1}\le2\varepsilon_k$, so any such sequence will do.
At stage $k$, with $f_{k-1}$ and $r_{k-1}$ already constructed, the remaining parameters are chosen in the order
\[
 \rho_k\;\longrightarrow\;N_0\;\longrightarrow\;M_k\;\longrightarrow\;r_k .
\]

The radius $\rho_k=2r_{k-1}/(1+r_{k-1})$ depends only on $r_{k-1}$ and determines the points $|a|\le\rho_k$ at which the moment estimates are needed.
The initial frequency $N_0$ depends on $f_{k-1}$, $r_{k-1}$, $\varepsilon_k$ and $\rho_k$; it keeps the new block small on the preceding disc $|z|\le r_{k-1}$.
The block length $M_k$ depends on all of the above, and produces the required moments.
Finally, $r_k=1-c/N_{M_k}$ is determined by the last frequency of the block.

\subsection{Proof of the main result}
\begin{proof}[Proof of Theorem~\ref{main:THM}]
\proofpart{1}{The inductive procedure:}
Fix amplitudes $0<\varepsilon_k\le1/100$, decreasing to zero, with $\varepsilon_{k-1}\le2\varepsilon_k$ for $k\ge2$, and put $f_0=0$ and $r_0=\tfrac12$.
Suppose $f_{k-1}$ and $r_{k-1}$ have been constructed, and put
\[
 \rho_k=\frac{2r_{k-1}}{1+r_{k-1}}.
\]
Apply Proposition~\ref{PROP:blocks} with initial polynomial
$Q_0=f_{k-1}$, amplitude $\varepsilon_k$, and $\rho=\rho_k$.
Choose an admissible odd initial frequency $N_0$, at least as large
as the last frequency of the preceding block, so that \eqref{tail} gives
\[
 \sup_{|z|\le r_{k-1}}
 |e^{\pm Q_j(z)}-e^{\pm Q_{j-1}(z)}|
 \le C\varepsilon_k r_{k-1}^{N_j}
 \le2^{-k-j},\qquad j\ge1.
\]
One choice of $N_0$ suffices, since $N_j=\kappa^jN_0\ge(j+1)N_0$.

For a block length $M_k$, still to be fixed, set
\[
 f_k(z)=Q_{M_k}(z)
     =f_{k-1}(z)+\sum_{j=1}^{M_k}z^{N_j}P_j(z),
 \qquad A_k=M_k\varepsilon_k^2.
\]
The frequencies are odd and the damping polynomials are even,
so $f_k$ remains odd. For all sufficiently large $M_k$,
Proposition~\ref{PROP:blocks} gives
\begin{equation}\label{block-moments}
 \int_\T e^{2\Re f_k}P_a\,dm\gtrsim A_k,
 \qquad
 \int_\T e^{4\Re f_k}P_a\,dm\lesssim A_k^2,
 \qquad |a|\le\rho_k.
\end{equation}
Our choice of $N_0$ also gives
\begin{equation}\label{preserve-disc}
 \sup_{|z|\le r_{k-1}}
 |e^{\pm f_k(z)}-e^{\pm f_{k-1}(z)}|
 \le\sum_{j=1}^{M_k}2^{-k-j}\le2^{-k}.
\end{equation}

To locate the new circle, use \eqref{EQ:block-derivative} and put
\[
 r_k=1-\frac{c}{N_{M_k}},
\]
where $c>0$ is a sufficiently small absolute constant. For $\zeta\in\T$, integrating along the radius and applying \eqref{EQ:block-derivative} to $f_k=Q_{M_k}$ gives
\[
 |e^{f_k(r_k\zeta)}-e^{f_k(\zeta)}|
 =\Bigl|\int_{r_k}^1(e^{f_k})'(s\zeta)\,\zeta\,ds\Bigr|
 \le(1-r_k)\norm{(e^{f_k})'}_\infty
 \lesssim\frac{c}{N_{M_k}}\,\varepsilon_kN_{M_k}
 =c\,\varepsilon_k,
\]
so our choice of $c$ ensures
\begin{equation}\label{circle-approximation}
 \sup_{\zeta\in\T}|e^{f_k(r_k\zeta)}-e^{f_k(\zeta)}|
 \le\tfrac12.
\end{equation}
Since $N_{M_k}=\kappa^{M_k}N_0$ and $N_0$ is fixed, taking $M_k$
sufficiently large also gives
\begin{equation}\label{choice-scales}
 \log\frac e{1-r_k}
 =\log\frac{eN_0}{c}+M_k\log \kappa
 \le(1+\log \kappa)M_k.
\end{equation}
We now fix $M_k$ large enough for these estimates and for
$r_k>(1+r_{k-1})/2$, choosing the block lengths so that
$A_k\to\infty$. This completes the inductive step and ensures
$r_k\uparrow1$.

\proofpart{2}{Passing to the limit:}
Fix $m>k$, and let $j$ run through the successive terms in blocks
$k+1,\ldots,m$. Summing \eqref{tail} over $j$ and using the
lacunarity of $N_j$ gives
\[
 \begin{aligned}
 (1-|z|)\bigl|(e^{\pm f_m}-e^{\pm f_k})'(z)\bigr|
 &\lesssim\varepsilon_{k+1}(1-|z|)
       \sum_j N_j|z|^{N_j-1}\\
 &\lesssim\varepsilon_{k+1}(1-|z|)
       \sum_{n=0}^\infty|z|^n
 =\varepsilon_{k+1}.
 \end{aligned}
\]
The comparison uses the disjoint ranges $N_j/2\le n<N_j$,
where $n\asymp N_j$.
All differences vanish at the origin, so
\[
 \norm{e^{\pm f_m}-e^{\pm f_k}}_\B
 \lesssim\varepsilon_{k+1}\longrightarrow0.
\]
Each exponential extends across $\T$. Thus there are $F,G\in\B_0$ with
\[
 e^{f_k}\longrightarrow F,
 \qquad e^{-f_k}\longrightarrow G
 \quad\hbox{in Bloch norm}.
\]
The convergence is locally uniform, and passing to the limit in the
identities for the approximants gives
\[
 F(z)G(z)=1,\qquad F(-z)=G(z)=F(z)^{-1}.
\]
In particular, $F$ is zero-free. Its analytic logarithm normalized to
vanish at the origin is odd and is the locally uniform limit of $f_k$.

\proofpart{3}{Moment estimates and visibility:}
By \eqref{preserve-disc},
\[
 \sup_{|z|\le r_k}|F(z)-e^{f_k(z)}|
 \le\sum_{\ell=k+1}^\infty2^{-\ell}=2^{-k}.
\]
Together with \eqref{circle-approximation}, this gives
\[
 \sup_{\zeta\in\T}|F(r_k\zeta)-e^{f_k(\zeta)}|
 \le2^{-k}+\tfrac12\le1.
\]
Since $r_k>(1+r_{k-1})/2$, we have $|z|/r_k\le\rho_k$ for
$|z|\le r_{k-1}$. The triangle inequality in $L^2$ and $L^4$,
applied to the probability measure $P_{z/r_k}\,dm$, therefore
transfers \eqref{block-moments} to $F$:
\begin{equation}\label{final-moments}
 \begin{aligned}
 \int_\T|F(r_k\zeta)|^2P_{z/r_k}(\zeta)\,dm(\zeta)&\gtrsim A_k,\\
 \int_\T|F(r_k\zeta)|^4P_{z/r_k}(\zeta)\,dm(\zeta)&\lesssim A_k^2,
 \qquad |z|\le r_{k-1},
 \end{aligned}
\end{equation}
for all sufficiently large $k$.
The Paley--Zygmund argument in Subsection~\ref{SEC:GUIDE} gives thresholds $L_k\asymp\sqrt{A_k}$ and an absolute constant $\gamma>0$ such that
\[
 \omega\left(\frac z{r_k},\D,
 \{\zeta\in\T:|F(r_k\zeta)|\ge L_k\}\right)
 \ge\gamma,\qquad |z|\le r_{k-1}.
\]
Since the block lengths were chosen so that $A_k\to\infty$, we have $L_k\to\infty$, and since $A_k=M_k\varepsilon_k^2$, \eqref{choice-scales} gives the remaining comparison in the criterion,
\[
 \varepsilon_k\sqrt{\log\frac e{1-r_k}}\lesssim\varepsilon_k\sqrt{M_k}\lesssim L_k.
\]
By the choice of the amplitudes in Step~1, the remaining hypotheses on $\varepsilon_k$ in Theorem~\ref{THM:SUFFNONCYCLIC} hold as well.
Theorem~\ref{THM:SUFFNONCYCLIC} now yields
$\operatorname{ind}[F,F^{-1}]=2$, proving Theorem~\ref{main:THM}.
\end{proof}

\begin{remark}[Weighted little Bloch spaces]\label{REM:weighted}
The proof of Theorem~\ref{main:THM} carries over to $\B_0(w)$ whenever $\Phi_w(r)\to\infty$, and the criterion already covers this setting, see Section~\ref{SEC:general}. In the construction, the inserted term at frequency $N_j$ in the $k$th block is given the amplitude $\varepsilon_k\,w(1/N_j)$. Then \eqref{block-moments} holds with
\[
 A_k=\varepsilon_k^2\sum_{j=1}^{M_k}w(1/N_j)^2
\]
in place of $A_k=M_k\varepsilon_k^2$. Since $N_j=\kappa^jN_0$ grows geometrically, this sum is comparable to a piece of the integral defining $\Phi_w$,
\[
 \sum_{j=1}^{M_k}w(1/N_j)^2\asymp\int_{1/N_{M_k}}^{1/N_0}w(t)^2\,\frac{dt}{t}.
\]
Hence long blocks give $A_k\to\infty$, precisely because $\Phi_w(r)\to\infty$.

The assumption that $w(t)/t^{1-\delta}$ is nonincreasing makes the weights at different scales comparable, which allows the estimates to be moved from one scale to another. The remaining arguments need only notational changes, with $\Phi_w(r_k)$ in place of $\log\frac e{1-r_k}$.

\end{remark}

\bibliographystyle{siam}
\bibliography{Figures/references}

\end{document}